\documentclass{article}

\usepackage{latexsym,amsmath,amsfonts,amscd,amsthm,upref,graphics}

\newtheorem{numbered}{}[section]
\newtheorem{Theorem}[numbered]{Theorem}

\newtheorem{nothing*}[numbered]{}
\newtheorem{remark}[numbered]{Remark}

\newtheorem{question}[numbered]{Question}
\newtheorem{Lemma}[numbered]{Lemma}
\newtheorem{Proposition}[numbered]{Proposition}
\newtheorem{Corollary}[numbered]{Corollary}

\begin{document}

\title{On Certain Projections of $C^*$-Matrix Algebras}
 \author{Ahmed Al-Rawashdeh\\ Department of Mathematical Sciences\\
 UAEU \footnote{United Arab Emirates University,
P.O. Box 17551,  Al-Ain UAE, Fax: +971-3-7671291, Phone:
+971-3-7136388.
}\\
 aalrawashdeh@uaeu.ac.ae}
\date{}
  \maketitle
\begin{abstract} H. Dye defined the projections $P_{i,j}(a)$ of a
$C^*$-matrix algebra by
\begin{eqnarray*}
P_{i,j}(a) &=& (1+aa^*)^{-1}\otimes E_{i,i} + (1+aa^*)^{-1}a \otimes
E_{i,j}\\
 &+& a^*(1+aa^*)^{-1} \otimes  E_{j,i} +
a^*(1+aa^*)^{-1}a\otimes E_{j,j},
\end{eqnarray*}
and he used it to show that in the case of factors not of type
$I_{2n}$, the unitary group determines the algebraic type of that
factor. We study these projections and we show that in
$\mathbb{M}_2(\mathbb{C})$, the set of such projections includes all
the projections. For infinite $C^*$-algebra $A$, having a system of
matrix units, including the Cuntz algebra $\mathcal{O}_n$, we have
$A\simeq \mathbb{M}_n(A)$. M. Leen proved that in a simple, purely
infinite $C^*$-algebra $A$, the $*$-symmetries generate
$\mathcal{U}_0(A)$. We revise and modify Leen's proof to show that
part of such $*$-isometry factors are of the form $1-2P_{i,j}(\omega
),\ \omega \in  \mathcal{U}(A)$. In simple, unital purely infinite
$C^*$-algebras having trivial $K_1$-group, we prove that all
$P_{i,j}(\omega )$ have trivial $K_0$-class. In particular, if $u\in
\mathcal{U}(\mathcal{O}_n)$, then $u$ can be factorized as a product
of $*$-symmetries, where eight of them are of the form
$1-2P_{i,j}(\omega )$. \vspace{.5cm}

\noindent \textbf {Keywords}:  $C^*$-algebras; $K_0$-class.\\
\noindent \textbf{MSC2010}: 46L05; 46L80.

\end{abstract}

\section{Introduction}

Let $A$ be a unital $C^*$-algebra. The set of projections and the
group of unitaries of $A$ are denoted by $\mathcal{P}(A)$ and
$\mathcal{U}(A)$, respectively. Recall that the $C^*$-matrix algebra
over $A$ which is denoted by $\mathbb{M}_n(A)$ is the algebra of all
$n\times n$ matrices $(a_{i,j})$ over $A$, with the usual addition,
scalar multiplication, and multiplication of matrices and the
involution (adjoint) is $(a_{i,j})^*= (a_{j,i}^*)$. As in Dye's
viewpoint of $\mathbb{M}_n(A)$, let $S_n(A)$ denote the direct sum
of $n$ copies of $A$, considered as a left $A$-module. Addition of
$n$-tuples $\bar{x}=(x_1, x_2, \ldots , x_n)$ in $S_n(A)$ is
componentwise and $a\in A$ acts on $\bar{x}$ by $a(\bar{x})=(ax_1,
ax_2, \ldots, ax_n)$. Then $S_n(A)$ is a Hilbert $C^*$-algebra
module, with the inner product defined by
\[<\bar{x},\bar{y}>= \sum_{i=1}^n x_iy_i^*.\]

By an $A$-endomorphism $T$ of $S_n(A)$, we mean an additive mapping
on $S_n(A)$ which commutes with left multiplication:
$a(\bar{x}T)=(a\bar{x})T$. In a familiar way, assign to any $T$ a
uniquely determined matrix $(t_{ij})$ over $A$ ($1\leq i,j\leq n$)
so that $\bar{x}T=(\sum_i x_it_{i1}, \ldots , \sum_i x_it_{in})$.

If $p$ is a projection in $\mathbb{M}_n(A)$, then $p$ is a mapping
on $S_n(A)$ having its range as a sub-module of $S_n(A)$. Then two
projections are orthogonal means their sub-module ranges are so. The
$C^*$-algebra $\mathbb{M}_n(A)$ contains numerous projections. For
each $a\in A$ and each pair of indices $i,j(i\neq j, \ 1\leq i,j
\leq n)$, H. Dye in \cite{Dye} defined the projection $P_{i,j}(a)$
in $\mathbb{M}_n(A)$, whose range consists  of all left multiples of
the vector with $1$ in the $i^{th}$-place, $a$ in the $j^{th}$-place
and zeros elsewhere. As a matrix

\[ P_{i,j}(a) =  \left( \begin{array}{ccccccc}
 0 & \cdots & \cdots & \cdots & \cdots & \cdots & 0 \\
     \multicolumn{7}{c}\dotfill \\
    0 & \cdots &  (1 + aa^{*})^{-1} & \cdots &  (1 + aa^{*})^{-1}a &  \cdots & 0 \\
       \multicolumn{7}{c}\dotfill \\
    0 & \cdots &  a^{*}(1 + aa^{*})^{-1} & \cdots &  a^{*}(1 + aa^{*})^{-1}a & \cdots & 0 \\
     \multicolumn{7}{c}\dotfill  \\
   0 & \cdots & \cdots & \cdots & \cdots & \cdots & 0
\end{array} \right)  \]

Recall that (see \cite{Dye}, p.74) a system of matrix units of a
unital $C^*$-algebra $A$ is a subset
 $\{e_{i,j}^{r}\}, 1\leq i,j \leq n$ and $1 \leq r\leq  m$ of $A$, such that
\[ e_{i,j}^{r} e_{j,k}^{r} = e_{i,k}^r,\ e_{i,j}^{r} e_{k,l}^{s}= 0\: \text{if}\:
r\neq\: s\: \text{or}\: j \neq k,\: (e_{i,j}^{r})^{*} =
e_{j,i}^{r},\: \sum_{i,r}^{n,m}e_{i,i}^{r}=1 \] and for every $i$,
$e_{i,i} \in \mathcal{P}(A)$. For the $C^*$-complex matrix algebra
$\mathbb{M}_n(\mathbb{C})$, let $\{E_{i,j}\}_{i,j=1}^{n}$ denote the
standard system of matrix units of the algebra, that is $E_{i,j}$ is
the $n\times n$ matrix over $\mathbb{C}$ with $1$ at the place
$i\times j$ and zeros elsewhere. It is also known that
$\mathbb{M}_n(A)$ is $*$-isomorphic to $A\otimes
\mathbb{M}_n(\mathbb{C})$ (see \cite{WO}). We will see that having a
system of matrix units is a necessary condition in order that a
$C^*$-algebra $A$ is $*$-isomorphic  to a $C^*$-matrix algebra
$\mathbb{M}_n(B)$. Using the notion of a system of matrix units, we
write
\begin{eqnarray*}
P_{i,j}(a) &=& (1+aa^*)^{-1}\otimes E_{i,i} + (1+aa^*)^{-1}a \otimes
E_{i,j}\\
 &+& a^*(1+aa^*)^{-1} \otimes  E_{j,i} +
a^*(1+aa^*)^{-1}a\otimes E_{j,j} \in \mathcal{P}(\mathbb{M}_n(A)).
\end{eqnarray*}
 If $a =0$, then $P_{i,j}(a)$ is the $i^{th}$ diagonal
matrix unit of $\mathbb{M}_n(A)$, which is $1\otimes E_{i,i}$, or
simply $E_i$.

\noindent Also in \cite{St}, M. Stone called the projection
$P_{i,j}(a)$ the
characteristics matrix of $a$.\\

H. Dye used these projections as a main tool to prove that an
isomorphism between the discrete unitary groups of von Neumann
factors not of type $I_n$, is implemented by a $*$-isomorphism
between the factors themselves [\cite{Dye}, Theorem 2]. Indeed, let
us recall main parts of his proof.
 Let $A$ and $B$ be two unital
$C^*$-algebras and let $\varphi : \mathcal{U}(A)\rightarrow
\mathcal{U}(B)$ be an isomorphism. As $\varphi$ preserves
self-adjoint unitaries, it induces a natural bijection
$\theta_{\varphi}:\mathcal{P}(A)\rightarrow \mathcal{P}(B)$ between
the sets of projections of $A$ and $B$ given by
\[ 1-2\theta_{\varphi}(p) = \varphi(1-2p),\;p\in \mathcal{P}(A).\]
This mapping is called a projection orthoisomorphism, if it
preserves orthogonality, i.e. $pq=0$ iff $\theta(p)\theta(q)=0$.

Now, let $\theta$ be an orthoisomorphism from
$\mathcal{P}(\mathbb{M}_n(A))$ onto $\mathcal{P}(\mathbb{M}_n(B))$.
In [\cite{Dye}, Lemma 8] when $A$ and $B$ are von Neumann algebras,
Dye proved that for any unitary $u\in \mathcal{U}(A)$,
$\theta(P_{i,j}(u))=P_{i,j}(v)$, for some unitary $v\in
\mathcal{U}(B)$. A similar result is proved in the case of simple,
unital $C^*$-algebras by the author in \cite{Ahmed}. Afterwards, Dye
in [\cite{Dye}, Lemma 6], proved that there exists a $*$-isomorphism
(or $*$-antiisomorphism) from $\mathbb{M}_n(A)$ onto
$\mathbb{M}_n(B)$ which coincides with $\theta$ on the projections
$P_{i,j}(a)$. In fact, he proved that $\theta$ induces the
$*$-isomorphism $\phi$ from $A$ onto $B$ defined by the relation
$P_{i,j}(a)=P_{i,j}(\phi (a))$.\\

 In this paper, we study the projections $P_{i,j}(a)$ of a
$C^*$-matrix algebra $\mathbb{M}_n(A)$, for some $C^*$-algebra $A$,
and we deduce main results concerning such projections.

The paper is organized as follows: In Section 2, we show that every
projection in $\mathbb{M}_2(\mathbb{C})$ is of the form
$P_{1,2}(a)$, for $a\in \mathbb{C}$. In Section 3, we show that some
infinite $C^*$-algebra $A$ is isomorphic to its matrix algebra
$\mathbb{M}_n(A)$, such as the Cuntz algebra $\mathcal{O}_n$, so the
projections $P_{i,j}(a)$ can be considered as projections of $A$.

 In a simple, unital purely infinite $C^*$-algebra $A$, M. Leen
proved that self-adjoint unitaries (also called $*$-symmetries, or
involutions) generate the connected component $\mathcal{U}_0(A)$ of
the unitary group $\mathcal{U}(A)$. Indeed, any unitary can be
written as a product of eleven $*$-symmetries. In Section 4, we
modify Leen's proof, and we write these $*$-symmetry factors
explicitly. By revising his proof and fixing some arbitraries using
a given system of matrix units, we show that eight of these
$*$-symmetry factors are in fact of the form $1-2P_{i,j}(\omega),\
\omega \in \mathcal{U}(A)$.\\

 Finally, in Section 5, we compute the $K_0$-class of such certain
projections, and we prove that in simple, unital purely infinite
$C^*$-algebras (assuming $K_1=0$), all projections of the form
$P_{i,j}(u),\ u\in \mathcal{U}(A)$ have trivial $K_0$-class. As a
good application for $\mathcal{O}_n$, we have that every unitary can
be written as a product of eleven $*$-symmetries (self-adjoint
unitaries, also called involutions), where eight of them are of the
form $1-2P_{i,j}(\omega ), \ \omega \in \mathcal{U}(\mathcal{O}_n)$.
Hence using \cite{Ahmed2} (Lemma 2.1), all such involutions of the
form $1-2P_{i,j}(\omega )$ are indeed conjugate, as group elements
in $\mathcal{U}(\mathcal{O}_n)$.

\section{The $2\times 2$-Complex Algebra Case}

Let $A$ be a unital $C^*$-algebra, and let $\mathcal{P}_{i,j}^n(A)$
denote the family of all projections in $\mathbb{M}_n(A)$ of the
form $P_{i,j}(a),\ 1\leq i,j \leq n,\ a\in A$. Also, let
$\mathcal{U}_{i,j}^n(A)$ denote the set of all self-adjoint
unitaries in $\mathbb{M}_n(A)$ of the form $1-2P_{i,j}(a),\ 1\leq
i,j \leq n,\ a\in A$. Notice that $\mathcal{P}_{i,j}^n(A)$ contains
non-trivial projections. In this small section, we show that in the
case of $\mathbb{M}_2(\mathbb{C})$, the set
$\mathcal{P}_{i,j}^2(\mathbb{C})$ includes all the non-trivial
projections $\mathcal{P}(\mathbb{M}_2(\mathbb{C}))$, i.e. every
non-trivial projection is of the form $P_{i,j}(a)$, for some complex
number $a$.

\begin{Proposition}
If $p\in \mathcal{P}(\mathbb{M}_2(\mathbb{C}))\backslash\{0,1\}$,
then $p\in \mathcal{P}_{i,j}^2(\mathbb{C})$.
\end{Proposition}
\begin{proof}
Let $p=\left( \begin{array}{cc}
    a & b
     \\ c & d
\end{array} \right ) $ be a non-trivial projection in
$\mathcal{P}(\mathbb{M}_2(\mathbb{C}))$. Then $a$ and $d$ are real
numbers. If $b=0$, then $p$ is either the diagonal matrix unit
$E_{1,1}$ or $E_{2,2}$. Otherwise, we have $a+b=1, a=a^2+ |b|^2$ and
$d=d^2+|b|^2$, therefore  $|b|^2\leq \frac{1}{4}$. By strightforward
computations, one can deduce that $p$ is of the form
$$ P_{1,2}\left (\frac{2b}{1+\sqrt{1-4|b|^2}} \right ), \  \text{or}\ \ \ P_{1,2}\left (\frac{2b}{1-\sqrt{1-4|b|^2}}\right ). $$
\end{proof}

\begin{remark} The projections in $\mathcal{P}_{i,j}^n(A)$ are all
of rank one by definition, this implies that in the case of
$\mathbb{M}_3(\mathbb{C})$, the set
$\mathcal{P}_{i,j}^3(\mathbb{C})$ does not cover all the non-trivial
projections. Indeed, there are projections in
$\mathcal{P}(\mathbb{M}_3(\mathbb{C}))$ of rank one which do not
belong to $\mathcal{P}_{i,j}^3(\mathbb{C})$, since every projection
in this latest family projects into a subspace of $\mathbb{C}^3$
which lies entirely in one coordinate plan.
\end{remark}
\section{Some Results for infinite $C^*$-algebras}

Let $A$ be a unital $C^*$-algebra having a system of matrix units
$\{e_{i,j}\}_{i,j=1}^n$, for some $n\geq 3$. Recall that
$e_{1,1}Ae_{1,1}$ is a $C^*$-algebra (corner algebra) which has
$e_{1,1}$ as a unit. This system of matrix units implements a
$*$-isomorphism between $A$ and $\mathbb{M}_n(e_{1,1}Ae_{1,1})$.
Indeed, let us define the mapping \[
\eta_{1}:\mathbb{M}_{n}(e_{1,1}Ae_{1,1})\rightarrow
 A \]by
 \[\eta_1((a_{i,j})^{n})=
\sum_{i,j=1}^{n}e_{i,1}a_{i,j}e_{1,j}.\] Moreover if $e_{1,1}$ is
equivalent to $1$ (i.e. $A$ is assumed to be infinite
$C^*$-algebra), then there exists a partial isometry $v$ of $A$ such
that $v^*v=e_{1,1}$ and $vv^*=1$, and this defines the
$*$-isomorphism  $\Delta_v: A \rightarrow e_{1,1}Ae_{1,1}$ by
$\Delta_v(x)= v^*xv$. The isomorphism $\Delta_v$ can be used to
decompose a projection as a sum of orthogonal equivalent
projections.
\begin{Proposition}\label{decom}
Let $A$ be a unital $C^*$-algebra having a system of matrix units
$\{e_{i,j}\}_{i=1}^n$. If $p$ is equivalent to the unity, then $p$
can be written as a sum of orthogonal equivalent subprojections.
\end{Proposition}
\begin{proof} As $p$ equivalent to 1, we consider the isomorphism
$\Delta_v$, then apply it to the equality $1= \sum_{i=1}^n e_{i,i}$,
to get $p=\sum_{i=1}^n v^*e_{i,i}v$. Then $p_i=v^*e_{i,i}v$, for all
$1\leq i \leq n$, are equivalent subprojections of $p$.
\end{proof}

Recall that, for two unital $C^*$-algebras $A$ and $B$, if $\alpha
:A\rightarrow B$ is a $*$-isomorphism, then $\alpha$ induces the
$*$-isomorphism $\widehat{\alpha}: \mathbb{M}_n(A) \rightarrow
\mathbb{M}_n(B)$, which is defined by $(a_{i,j})\mapsto
(\alpha(a_{i,j}))$. Then we have the following result.

\begin{Proposition}\label{matrixiso}
Let $A$ be an infinite unital $C^*$-algebra having a system of
matrix units $\{e_{i,j}\}_{i,j=1}^n$. If $e_{1,1}$ is equivalent to
$1$, then $\mathbb{M}_n(A)$ is $*$-isomorphic to $A$.
\end{Proposition}
\begin{proof} Let $\Delta_v: A \rightarrow e_{1,1}Ae_{1,1}$ and $
\eta_{1}:\mathbb{M}_{n}(e_{1,1}Ae_{1,1})\rightarrow A$ be defined as
above. Then the mapping $\eta =\eta_1 \circ \widehat{\Delta_v}$ is a
$*$-isomorphism from $\mathbb{M}_n(A)$ onto $A$. Moreover,
\[\eta (a_{i,j})^n = \sum_{i,j}^n e_{i,1}v^*a_{i,j}ve_{1,j}, \ \text{and}\]
\[\eta^{-1}(x)= (ve_{1,i}xe_{j,1}v^*)_{i,j}^n.\]
\end{proof}

As a main example of purely infinite $C^*$-algebras, let us recall
the Cuntz algebra $\mathcal{O}_{n}$; $n\geq 2$, is the universal
$C^{*}$-algebra  which is generated by isometries
$s_{1},s_{2},\ldots ,s_{n}$, such that
 $\sum_{i=1}^{n}s_{i}s_{i}^{*} = 1$ with $s_{i}^{*}s_{j}=0$, when $i\neq
 j$ and $s_i^*s_i=1$ (for more details, see \cite{Cu2}, [\cite{Da}, p.149]). Let
  \begin{equation}\label{sysmuCuntz} e_{i,j}=s_is_j^*,\ \ \ \ \ 1\leq i,j\leq
  n\ .
 \end{equation}
  Then $\{e_{i,j}\}_{i,j=1}^n$ forms a
 system of matrix units for $\mathcal{O}_n$. As $s_1^*$ partial
 isometry between $e_{1,1}$ and the unity, then Proposition
 \ref{matrixiso} shows that the mapping
 \begin{equation}\label{etamap}
 \eta : \mathbb{M}_n(\mathcal{O}_n)\rightarrow \mathcal{O}_n,\ \
 (a_{i,j})_{i,j}\mapsto \sum_{i,j=1}^n s_ia_{i,j}s_j^*
 \end{equation}
 is a $*$-isomorphism. Indeed, for $x\in \mathcal{O}_n$,
 $\eta^{-1}(x)= (s_i^*xs_j)_{i,j}\in \mathbb{M}_n(\mathcal{O}_n)$.

 \noindent Therefore, we have proved the following result, which is
 in fact known, but for sake of completeness:
 \begin{Proposition}\label{matrixcuntz}
 The Cuntz algebra $\mathcal{O}_n$ is isomorphic to the
 $C^*$-algebra  $\mathbb{M}_n(\mathcal{O}_n)$.
 \end{Proposition}

 Then for $a\in \mathcal{O}_n$, $P_{i,j}(a)$
 are considered as projections of $\mathcal{O}_n$ by applying the
 mapping $\eta$. Therefore,
 \[P_{i,j}(a) = s_i(1+aa^*)^{-1}s_i^* +
s_i(1+aa^*)^{-1}as_j^* +
 s_ja^*(1+aa^*)^{-1}s_i^* + s_ja^*(1+aa^*)^{-1}as_j^*.\]

\section{Unitary Factors in Purely Infinite $C^*$-Algebras}

Recall that in a unital $C^*$-algebra $A$, every self-adjoint
unitary $u$ ($*$-symmetry, or also called an involution) can be
written as $u=1-2p$, for some projection $p\in \mathcal{P}(A)$, let
us say " the self-adjoint unitary $u$ is associated to the
projection $p$". In this section, we assume that $A$ is purely
infinite simple $C^*$-algebra, and we study the factorizations of
unitaries of $A$. Recall that in \cite{Leen}, M. Leen proved that
every unitary in the connected component of the unity
$\mathcal{U}_0(A)$ is generated by $*$-symmetries.

 Consider a system of matrix units $\{e_{i,j}\}_{i,j=1}^n$ of $A$, with $e_{1,1}\sim 1$.
 Let us recall the $*$-isomorphisms \(
\eta_{1}:\mathbb{M}_{n}(e_{1,1}Ae_{1,1})\rightarrow
 A \), and $\eta =\eta_1 \circ \widehat{\Delta_v}$ from $\mathbb{M}_n(A)$ onto
 $A$. We modify Leens' proof of Theorem 3.5 in \cite{Leen} by revising his arguments, and then we prove the following main
theorem, which shows that every unitary of $A$ can be factorized as
a product of eleven self-adjoint unitaries ($*$-symmetries)
moreover, where eight of such factors are associated to the
projections $P_{i,j}(\mu)$, for some $\mu\in \mathcal{U}(A)$.
\begin{Theorem}\label{unitaryfactor}
Let $A$ be a simple, unital purely infinite $C^*$-algebra, such that
$K_1(A)=0$, and let $\{e_{i,j}\}_{i,j=1}^n$ be a system of matrix
units of $A$, with $e_{1,1}\sim 1$. Then every unitary $a$ of $A$
can be written as
\[a= z_1(\prod_{k=1}^4v_k)z_2z_3,\]
where $z_1,z_2,z_3$ are some self-adjoint unitaries and the
 $v_i's$ are the self-adjoint unitaries of $A$ defined by:
 \begin{eqnarray*}
 v_1 &=& [1-2\eta(P_{1,2}(-\alpha))][1-2\eta(P_{1,2}(-1))]\\
v_2 &=& [1-2\eta(P_{1,3}(-\alpha))][1-2\eta(P_{1,3}(-1))]\\
v_3 &=& [1-2\eta(P_{1,2}(-\gamma))][1-2\eta(P_{1,2}(-1))]\\
v_4 &=& [1-2\eta(P_{1,3}(-\gamma))][1-2\eta(P_{1,3}(-1))],
 \end{eqnarray*}
for some $\alpha,\gamma \in \mathcal{U}(A)$.
\end{Theorem}

Consequently, as the Cuntz algebra is simple, unital purely infinite
$C^*$-algebra, and $K_1(\mathcal{O}_n)=0$, see \cite{Cu1}, and using
Proposition \ref{matrixcuntz}, we have the following result.
\begin{Corollary}
If $u\in \mathcal{U}(\mathcal{O}_n)$, then
 \begin{eqnarray*}
u &=& z_1(1-2P_{1,2}(-\alpha))(1-2P_{1,2}(-1))(1-2P_{1,3}(-\alpha))(1-2P_{1,3}(-1)) \\
&.&
(1-2P_{1,2}(-\gamma))(1-2P_{1,2}(-1))(1-2P_{1,3}(-\gamma))(1-2P_{1,3}(-1))z_2z_3,
\end{eqnarray*}
for some self-adjoint unitaries $z_1,z_2,z_3$ and $\alpha, \gamma
\in \mathcal{U}(\mathcal{O}_n)$.
\end{Corollary}

Now, in order to prove out main theorem, let us recall the following
result of M. Leen.

\begin{Theorem}[\cite{Leen}, Theorem 3.8]\label{Leen}
Let $A$ be a simple, unital purely infinite $C^*$-algebra. Then the
$*$-symmetries (self-adjoint unitaries) generate the connected
component of the unity $\mathcal{U}_0(A)$.
\end{Theorem}

 So Leen proved that every unitary in the component of the unity, can be
 written as a finite product of self-adjoint unitaries. We shall use Leen's approach, indeed, we fix some
arbitrates, and we modify some of his arguments. Then using the
system of matrix units and the mappings $\eta_1,\ \eta$, we write
some arguments in an explicit way. Finally, we deduce that eight of
those self-adjoint unitaries, as factors, are in fact associated to
the projections $P_{i,j}(u)$, for some $u\in \mathcal{U}(A)$.

Let us introduce the following lemma which in fact, M. Leen used in
his proof, and we do in our proof as well.
\begin{Lemma}\label{unitarymatrixdecomp}
Let $A$ be a simple, unital purely infinite $C^*$-algebra, and let
$\rho$ be a non-trivial projections of $A$. If $a \in
\mathcal{U}_0(A)$, then there exist self-adjoint unitaries
$z_1,z_2,z_3$ of $A$ and $x\in \mathcal{U}_0(A)$ such that
\[z_1az_2z_3=  \left( \begin{array}{cc}
x & 0 \\
0 & 1- \rho
\end{array} \right).\]
\end{Lemma}
\begin{proof}
Mimic the first part of the proof of Theorem 3.5 in \cite{Leen},
with replacing symmetries by $*$-symmetries and invertible by
unitaries.
\end{proof}
\noindent \textbf{Proof of Theorem \ref{unitaryfactor}}:
\begin{proof} Since $A$ is a simple, unital purely infinite $C^*$-algebra, using \cite{Cu1}, we have $K_1(A) \simeq
\mathcal{U}(A)/\mathcal{U}_0(A)$. As $K_1(A)$ is assumed to be
trivial, we have $\mathcal{U}(A)=\mathcal{U}_0(A)$. Now suppose
$a\in \mathcal{U}(A)$, we shall revise Leen's proof, for many
details, we just refer to him, and we explain new arguments which
shall lead to our result.
  Let $p=e_{1,1}$, as $p\sim 1$, use Proposition \ref{decom} and the isomorphism $\Delta_u$ ($u^*u=e_{1,1}, uu^*=1$) to find
  a projection $p_1<p$ (precisely, $p_1=u^*e_{1,1}u$) which is equivalent to $p$
  moreover,
   set the partial isometry $v=u^*e_{1,1}$, and put $\rho = p -p_1$. Using
   Lemma \ref{unitarymatrixdecomp}, there exist self-adjoint
   unitaries $z_1,z_2$ and $z_3$ such that
\[ z_1az_2z_3 = \left( \begin{array}{cc}
x & 0 \\
0 & 1- \rho
\end{array} \right)  ,\] where $x \in \mathcal{U}(\rho A \rho
)$. We will show that the right hand side can be written as a
product of eight self-adjoint unitaries, each of them is associated
to a projection of the form $\eta P_{i,j}(u)$, for some $u\in
\mathcal{U}(A)$. We may replace $z_1az_2z_3$ by $a$.

  Choose $q=e_{2,2}$, $r=e_{3,3}$ and put $r_{1}=p + q +
 r$, then we have $q\sim r <1-p-q$.
  Let $v_{1}=e_{2,1},\ v_{2}=e_{3,2},\ \text{and} \
v_{3}=e_{1,3}$,
 so $v_1,v_2$ and $v_3$ are partial isometries such that
$$ v_{1}^{*}v_{1}=p,\ v_{1}v_{1}^{*}=q,\
 v_{2}^{*}v_{2}= q,\ v_{2}v_{2}^{*}= r,\
 v_{3}^{*}v_{3}= r,\ \text{and}\ v_{3}v_{3}^{*}= p .$$
 Let $w= v_1+v_2 + vv_3$. Recall that
  $\mathbb{K}$ denotes the compact operators on
the separable,
 infinite dimensional Hilbert space $\ell^2(\mathbb{N})$. By $I$ in $\rho A\rho \otimes
\mathbb{K}$ we mean $\rho \otimes 1_{\infty}(\mathbb{C})$.

   Leen defined in his proof three
isomorphisms: \( \rho A\rho \otimes \mathbb{K} \longrightarrow
r_{1}Ar_{1}\).
 In order to build the first
of the three copies of $\rho A\rho \otimes \mathbb{K}$, he defined
an infinite collection of projections using $w$ and $\rho$ as
follows: $\rho_k = w\rho_{k-1}w^*$, for $k\geq 2$, $\rho_1 =\rho$
and $w_k=w^{k-1}\rho$. Then $w_kw_k^*=\rho_k$ and $w_k^*w_k=\rho$,
the $\rho_k's$ are orthogonal equivalent projections which satisfy
$\rho_{3n-2}<p$, $\rho_{3n-1}<q$ and $\rho_{3n}<r$, for $n\geq 1$.

 Define $\chi : \rho A\rho \otimes \mathbb{K} \rightarrow
 r_{1}Ar_{1}$ by $y\otimes E_{i,j}(\mathbb{C})\mapsto w_iyw_j^*$, and $I\mapsto r_1$. Next we produce two other copies of $\rho A\rho \otimes
 \mathbb{K}$ in $r_{1}Ar_{1}$ as follows: For each $n$ choose
 orthogonal equivalent projections $\{e_{3n-2}^j: j=1,\ldots ,
 4^{n-1}\}$ such that $e_{3n-2}^j\sim \rho_{3n-2}$ and
 \[ \rho_{3n-2} = \sum_{j=1}^{4^{n}-1}e_{3n-2}^j ,\]
  then put $e_{3n-1}^j = w(e_{3n-2}^j)w^*$ and $e_{3n}^j =w(e_{3n-1}^j)w^*$,
   for each $n$ and $j$, and order the $e_i^j$'s
   as: $e_1^1,e_2^1,e_3^1,e_4^1,\ldots e_4^4,e_5^1,\ldots $. Use
  the partial isometries which implements the equivalences
   $\rho_{3n-2}\sim e_{3n-2}^j$ and $\rho_{3n-2}\sim \rho$ to define
  partial isometries $r_{3n-2}^j$ so that
  $r_{3n-2}^j(r_{3n-2}^j)^* = \rho$ and
  $(r_{3n-2}^j)^*r_{3n-2}^j=e_{3n-2}^j$, and put $r_{3n-1}^j=
  r_{3n-2}^jw^*$ and $r_{3n}^j= r_{3n-1}^jw^*$. Then use the
  $r_i^j$ to define $\varphi_1 : \rho A\rho \otimes \mathbb{K} \rightarrow
 r_{1}Ar_{1}$.

Similarly choose orthogonal equivalent projections $\{f_i^j\}$ such
that $\rho =f_1^1$ and
\[ \rho_{3n-1} = \sum_{j=1}^{2.4^{n-1}}f_{3n-1}^j\:,\]
for $n\geq 1$. Then put $f_{3n}^j= w(f_{3n-1}^j)w^*$ and $f_{3n+1}^j
=w(f_{3n}^j)w^*$, for any $n$ and $j$. Order the $f_i^j$ as:
\[ f_1^1,f_2^1,f_2^2,f_3^1f_3^2,f_4^1,f_4^2,f_5^1,\ldots ,
f_5^8,f_6^1,\ldots .\] Using the partial isometries which implement
$f_i^j \sim\rho$, define $\varphi_2 :\: \rho A\rho \otimes
\mathbb{K} \rightarrow
 r_{1}Ar_{1}$.

 \noindent Recall that $w= e_{2,1} +e_{3,2}+u^{*}e_{1,3}$, then
\[ w^{2}=
e_{2,1}u^{*}e_{1,3}+e_{3,1}+e_{3,2}u^{*}e_{1,3}+u^{*}e_{1,2}+u^{*}e_{1,3}u^{*}e_{1,3}
\]
Now for $1 \leq k \leq 3$, let $u_{k}= w^{k-1}p$ therefore
$u_{k}=e_{k,1}$. Define the map
\[ \zeta : r_{1}Ar_{1} \longrightarrow \mathbb{M}_{3}(pAp) \]
\[ \text{by} \: x \longmapsto (u_{i}^{*}xu_{j})_{i,j=1}^{3} \]
\[ \text{i.e.} \: x \longmapsto (e_{1,i}xe_{j,1})_{i,j=1}^{3} .\]
The map $\zeta$ is a $*$-isomorphism, indeed
\[ \zeta ^{-1} : \mathbb{M}_{3}(pAp) \longrightarrow r_{1}Ar_{1} \]
\[ \text{is defined by}\: (a_{i,j}) \longmapsto
\sum_{i,j}^{3} e_{i,1}a_{i,j}e_{1,j}. \]

Now we turn to factorization of $a$ (In fact, we factorize
$z_1az_2z_3$). Let $r_{0}= 1-r_{1}$. From the definitions of
$\varphi_{i}$'s, and since
 $a = \left(
\begin{array}{cc}
x & 0 \\
0 & 1- \rho
\end{array} \right)$,  where $x \in \mathcal{U}(\rho A \rho )$,
 we have the following:
\begin{eqnarray*}
\varphi_{i}(\text{diag}(x,1,1,\ldots ))\:&=&\:
\varphi_{i}(\text{diag}(x-\rho, 0,0 , \ldots
)+ I)\\
&=& r_{1} + \varphi_{i}(\text{diag}(x-\rho,0,0,\ldots ))\\
&=& r_{1} + x-\rho \\
&=& p + q+ r +x-\rho\\
&=& a-r_{0}.
\end{eqnarray*}
If $a-r_0$ is a product of $*$-symmetries in $r_1Ar_1$, then $a$ is
a product of $*$-symmetries in $A$. Using [\cite{Leen}, proof of Theorem 1], we factorize $\text{diag}(x,1,1,\ldots )$ as follows:\\
\( \text{diag}(x,1,1,\ldots ) =\\
\text{diag}(x^{1/2},x^{-1/2},1,x^{1/8},x^{1/8},x^{1/8},
x^{1/8},x^{-1/8},x^{-1/8},x^{-1/8},x^{-1/8},1,1,1,1,\ldots )\\
. \text{diag}(x^{1/2},1,x^{-1/2},x^{1/8},x^{1/8},x^{1/8},x^{1/8},1,1,1,1,x^{-1/8},x^{-1/8},x^{-1/8},x^{-1/8},\ldots )\\
.
\text{diag}(1,x^{1/4},x^{1/4},x^{-1/4},x^{-1/4},1,1,x^{1/16},x^{1/16},x^{1/16},x^{1/16},x^{1/16},x^{1/16},
x^{1/16},x^{1/16},x^{-1/16},\ldots )\\
.
\text{diag}(1,x^{1/4},x^{1/4},1,1,x^{-1/4},x^{-1/4},x^{1/16},x^{1/16},x^{1/16},x^{1/16},x^{1/16},x^{1/16},x^{1/16},x^{1/16}
,\ldots)\\
= b_{1}b_{2}b_{3}b_{4}. \)\\

 We must factorize $b_i$ as a product of $*$-symmetries. We use
$\varphi_1$ to factorize $b_1$, $b_2$, and use $\varphi_2$ to
factorize $b_3$, $b_4$. We check the details only for $b_1$ and
$b_2$.

\noindent Let us first factorize $b_{1}$.
\[ b_{1}= \:(b_{1}^{1},b_{1}^{2}, \ldots , b_{1}^{n}, \ldots );\]
 where $b_{1}^{n}= \text{diag}(x_{n},x_{n}^{-1},1)$ and $x_{n}$ be the diagonal $4^{n-1}\times 4^{n-1}$
 matrix with all diagonal entries equal to $x^{(\frac{1}{2.4^{n-1}})}$
, so $b_{1} \in\:
\prod_{n=1}^{\infty}\mathbb{M}_{3}(\mathbb{M}_{4^{n-1}}(\rho A\rho
))$. Then Leen defined the map
\[ \Phi : \prod_{n=1}^{\infty}\mathbb{M}_{3}(\mathbb{M}_{4^{n-1}}(\rho A\rho )) \longrightarrow \prod_{n=1}^{\infty}\mathbb{M}_{3}(\rho A\rho) \]
Let $\Phi (b_{1})= c^{1}$. He showed that $\chi (c^{1})=
\varphi_{1}(b_{1})$, and
\[ \zeta (\chi (c^{1}))= \left( \begin{array}{ccc}
\alpha & 0 & 0 \\
0 & \alpha^{-1} & 0 \\
0 & 0  & p
\end{array} \right); \]
 where $\alpha$ is a unitary in $pAp$.
Let $ \beta_{1} =\left( \begin{array}{ccc}
0 & \alpha & 0 \\
\alpha^{-1} & 0 & 0 \\
0 & 0  & p
\end{array} \right) $ and $\beta_{2} = \left( \begin{array}{ccc}
0 & p & 0 \\
p & 0 & 0 \\
0 & 0  & p
\end{array} \right), $
 so $\beta_{1} \beta_{2} = \zeta (\chi (c^{1}))$ and
\[\frac{I-\beta_{1}}{2}= P_{1,2}(-\alpha),\:\:\:\frac{I-\beta_{2}}{2}= P_{1,2}(-p), \]
 where now $P_{1,2}(-\alpha),P_{1,2}(-p)\:\:\in \mathcal{P}(\mathbb{M}_{3}(pAp))$.
Therefore,
\[ \chi (c^{1})= \zeta^{-1}(\beta_{1})\zeta^{-1}(\beta_{2})
=(r_{1}-2\zeta^{-1}(P_{1,2}(-\alpha)))(r_{1}-2\zeta^{-1}(P_{1,2}(-p))),
\] but $\zeta^{-1}(P_{1,2}(-\alpha))=
\eta_{1}(P_{1,2}(-\alpha))$ and
$\zeta^{-1}(P_{1,2}(-p)= \eta_{1}(P_{1,2}(-p))$.\\
Now to factorize $b_{2}$:
\[ b_{2}= (b_{2}^{1},b_{2}^{2},\ldots ,b_{2}^{n}, \ldots )\:\text{where}\:
b_{2}^{n}= \text{diag}(x_{n},1,x_{n}^{-1}) \] and $x_{n}$ is the
same as in $b_{1}$ so $b_{2}\in\:
\prod_{n=1}^{\infty}\mathbb{M}_{3}(\mathbb{M}_{4^{n-1}}(\rho
A\rho))$. Let $\Phi (b_{2})= c^{2}$. $\chi (c^{2})\:=\:
\varphi_{1}(b_{2})$
\[\zeta (\chi (c^{2}))\:=\: \left( \begin{array}{ccc}
\alpha & 0 & 0 \\
0 & p & 0 \\
0 & 0  & \alpha^{-1}
\end{array} \right) \:=\:  \left( \begin{array}{ccc}
0 & 0 & \alpha \\
0 & p & 0 \\
\alpha^{-1} & 0  & 0
\end{array} \right) \left( \begin{array}{ccc}
0 & 0 & p \\
0 & p & 0 \\
p & 0  & 0
\end{array} \right)=\beta_{3}\beta_{4} \]
so $\beta_{3},\:\beta_{4}$ are self-adjoint unitaries in
$\mathbb{M}_{3}(pAp)$, indeed
\[ \frac{I-\beta_{3}}{2}= P_{1,3}(-\alpha),\:\text{and}\: \frac{I-\beta_{4}}{2}=P_{1,3}(-p) \]
therefore,
\[ \chi (c^{2})= \zeta^{-1}(\beta_{3})\zeta^{-1}(\beta_{4})=
(r_{1}-2\zeta^{-1}(P_{1,3}(-\alpha)))(
r_{1}-2\zeta^{-1}(P_{1,3}(-p))) \] but
$\zeta^{-1}(P_{1,3}(-\alpha))=
\eta_{1}(P_{1,3}(-\alpha)),\:\text{and} \:
\zeta^{-1}(P_{1,3}(-p))=\eta_{1}(P_{1,3}(-p)) $.\\

\noindent Now we use $\varphi_2$ to factorize $b_{3}$ and $b_4$:
\[ b_{3}= (1,b_{3}^{1},b_{3}^{2},\ldots ,b_{3}^{n},\ldots );\:\text{where}
\:b_{3}^{n}=\text{diag}(x_{n},x_{n}^{-1},1) \] and $x_{n}$ is a
$2.4^{n-1}\times 2.4^{n-1}$  diagonal  of diagonal entries matrix
 $x^{\frac{1}{4.4^{n-1}}}$\\
so $b_{3} \in (\rho A\rho)\times (\prod
\mathbb{M}_{3}(\mathbb{M}_{2.4^{n-1}}(\rho A\rho))) $. Then we
define the map
\[ \Phi^{\prime} :(\rho A\rho)\times (\prod \mathbb{M}_{3}(\mathbb{M}_{2.4^{n-1}}(\rho A\rho))) \longrightarrow (\rho A\rho)\otimes \mathbb{K}, \]
which acts as the identity map on the first component. Let
$\Phi^{\prime}(b_{3})= d^{1}$. We have $\chi (d^{1})=
\varphi_{2}(b_{3})$.
\[ \zeta (\chi (d^{1})= \left( \begin{array}{ccc}
\gamma & 0 & 0 \\
0 & \gamma^{-1} & 0 \\
0 & 0 & p
\end{array} \right); \] where $\gamma$ is a unitary in $\rho A\rho$
, so similar to case $b_{1}$, just replace $\alpha$ by $\gamma$, to
get
\[ \chi (d^{1})= (r_{1}-2\eta_{1}(P_{1,2}(-\gamma)))(r_{1}-2\eta_{1}(P_{1,2}(-p))) .\]
Now finally to factorize $b_{4}$:
\[ b_{4}= \text{diag}(1,b_{4}^{1},b_{4}^{2},\ldots , b_{4}^{n},\ldots );\:\text{where} \:b_{4}^{n}\:=\:
\text{diag}(x_{n},1,x_{n}^{-1}), \] and $x_{n}$ is the same as in
the case of $b_{3}$ Let $\Phi^{\prime}(b_{4})=d^{2}$. We have $\chi
(d^{2})=\varphi_{2}(b_{4})$.
\[ \zeta (\chi (d^{2})= \left( \begin{array}{ccc}
\gamma & 0 & 0 \\
0 & p & 0 \\
0 & 0 & \gamma^{-1}
\end{array} \right) \]
again, it's similar to case $b_{2}$, so
\[ \chi (d^{2})= (r_{1}-2\eta_{1}(P_{1,3}(-\gamma)))(r_{1}-2\eta_{1}(P_{1,3}(-p))) .\]

\noindent Then, we factorize $a-r_{0}$ as
\[ a-r_{0} = \chi (c^{1})\chi (c^{2})\chi (d^{1})\chi (d^{2}) \]
therefore,
\[ \left( \begin{array}{cc}
a-r_{0} & 0 \\
0 & r_{0} \end{array} \right) = \left( \begin{array}{cc}
\chi (c^{1}) & 0 \\
0 & r_{0} \end{array} \right) \left( \begin{array}{cc}
\chi (c^{2}) & 0 \\
0 & r_{0} \end{array} \right) \left( \begin{array}{cc}
\chi (d^{1}) & 0 \\
0 & r_{0} \end{array} \right) \left( \begin{array}{cc}
\chi (d^{2}) & 0 \\
0 & r_{0} \end{array} \right). \] And also we have the following:
\[ \left( \begin{array}{cc}
\chi (c^{1}) & 0 \\
0 & r_{0} \end{array} \right) =\left( \begin{array}{cc}
r_{1}-2\eta_{1}(P_{1,2}(-\alpha)) & 0 \\
0 & r_{0} \end{array} \right) \left( \begin{array}{cc}
r_{1}-2\eta_{1}(P_{1,2}(-p)) & 0 \\
0 & r_{0} \end{array} \right) \]
\[ \left( \begin{array}{cc}
\chi (c^{2}) & 0 \\
0 & r_{0} \end{array} \right) =\left( \begin{array}{cc}
r_{1}-2\eta_{1}(P_{1,3}(-\alpha)) & 0 \\
0 & r_{0} \end{array} \right) \left( \begin{array}{cc}
r_{1}-2\eta_{1}(P_{1,3}(-p)) & 0 \\
0 & r_{0} \end{array} \right) \]

\[ \left( \begin{array}{cc}
\chi (d^{1}) & 0 \\
0 & r_{0} \end{array} \right) =\left( \begin{array}{cc}
r_{1}-2\eta_{1}(P_{1,2}(-\gamma)) & 0 \\
0 & r_{0} \end{array} \right) \left( \begin{array}{cc}
r_{1}-2\eta_{1}(P_{1,2}(-p)) & 0 \\
0 & r_{0} \end{array} \right) \]
\[ \left( \begin{array}{cc}
\chi (d^{2}) & 0 \\
0 & r_{0} \end{array} \right) =\left( \begin{array}{cc}
r_{1}-2\eta_{1}(P_{1,3}(-\gamma)) & 0 \\
0 & r_{0} \end{array} \right) \left( \begin{array}{cc}
r_{1}-2\eta_{1}(P_{1,3}(-p)) & 0 \\
0 & r_{0} \end{array} \right) . \]
 Therefore,
\begin{eqnarray*}
z_1az_2z_3 &=& (1-2\eta_{1}(P_{1,2}(-\alpha)))(1-2\eta_{1}(P_{1,2}(-p)))(1-2\eta_{1}(P_{1,3}(-\alpha)))(1-2\eta_{1}(P_{1,3}(-p))) \\
&.&
(1-2\eta_{1}(P_{1,2}(-\gamma)))(1-2\eta_{1}(P_{1,2}(-p)))(1-2\eta_{1}(P_{1,3}(-\gamma)))(1-2\eta_{1}(P_{1,3}(-p)))
\end{eqnarray*}
 The factors in the right side are all self-adjoint unitaries in $A$. Hence using the
 mapping $\eta$, we have that
 \begin{eqnarray*}
a &=& z_1(1-2\eta(P_{1,2}(-\alpha)))(1-2\eta(P_{1,2}(-1)))(1-2\eta(P_{1,3}(-\alpha)))(1-2\eta(P_{1,3}(-1))) \\
&.&
(1-2\eta(P_{1,2}(-\gamma)))(1-2\eta(P_{1,2}(-1)))(1-2\eta(P_{1,3}(-\gamma)))(1-2\eta(P_{1,3}(-1)))z_2z_3
\end{eqnarray*}
where $\alpha$ and $\gamma$ are unitaries in $A$, and this ends the
proof.
\end{proof}

\noindent Finally, let us finish this section by the following open
question:
\begin{question}\label{openQ}
In the Cuntz algebra $\mathcal{O}_n$, do self-adjoint unitaries of
the form $\{1-2P_{i,j}(a)\}$ generate the unitary group
$\mathcal{U}(\mathcal{O}_n)?$
\end{question}

\section{$K$-Theory of Certain Projections}
In this section, we study the $K_0$-class of the projections
$P_{i,j}(u)$, where $u$ is a unitary of some unital $C^*$-algebra
$A$. In particular, if $A$ is a simple purely infinite
$C^*$-algebra, with $K_1(A)=0$, or $A$ is a von Neumann factor of
type $II_1$, or $III$, then for any unitary $u$ of $A$, $P_{i,j}(u)$
has trivial $K_0$-class. Afterwards, we present an application of
Theorem \ref{unitaryfactor}, to the case of Cuntz algebras.
\begin{Proposition}\label{pijclass1}
Let $A$ be a unital $C^*$-algebra. If $v$ is a unitary in $A$ of
finite order, then $[P_{i,j}(v)]= [1]$ in $K_0(A)$.
\end{Proposition}
\begin{proof} Consider a unitary $v$ in $A$, such that $v^{m}=1$, for some positive integer $m$.
For $i\neq j$, let $$W= \frac{1}{\sqrt{2}}(v\otimes E_{i,i} +
v\otimes E_{i,j}+E_{j,i} - E_{j,j}+\sum_{k\notin \{i,j\}}
\sqrt{2}\otimes E_{k,k})\ ,$$ then
$W^*=\frac{1}{\sqrt{2}}(v^{m-1}\otimes E_{i,i} +
E_{i,j}+v^{m-1}\otimes E_{j,i} - E_{j,j}+\sum_{k\notin \{i,j\}}
\sqrt{2}\otimes E_{k,k}) ,$ therefore $W \in
\mathcal{U}(\mathbb{M}_{n}(A))$. Moreover,

\begin{eqnarray*}
W^*P_{i,j}(v) W &=& \frac{1}{4}(2v^{m-1}\otimes E_{i,i} +2\otimes
E_{i,j})(\sqrt{2}W)\\
&=& \left( \begin{array}{cccccc}
    0 & 0 & \cdots  & 0 & \cdots & 0 \\
    \vdots &  \vdots  & \ddots & \vdots &\ddots & \vdots   \\
    0 & 0   & \cdots & 1 &\cdots & 0\\
    \vdots &  \vdots  & \ddots & \vdots &\ddots & \vdots   \\
     0 & 0 & \cdots  & 0 & \cdots & 0
\end{array} \right)\  (\text{1 at the i-th place})\\
 &=& E_{i,i}.
\end{eqnarray*}
 This implies that the projection $P_{i,j}(v)$ is unitarily
 equivalent to $E_{i,i}$ in $\mathbb{M}_n(A)$, therefore we have that $[P_{i,j}(v)]=[1]$ in
 $K_0(A)$, hence the proposition has been checked.
\end{proof}

\begin{Proposition}\label{pijclass2} Let $A$ be a unital $C^*$-algebra. If $w_1,w_2$
and $v$ are unitaries of $A$ such that $v$ has order $m$, then
$[P_{i,j}(w_1vw_2)]= [1]$ in $K_0(A)$.
\end{Proposition}
\begin{proof} As $w_1$ and $w_2$ are unitaries in $A$, then for all
$i\neq j$, $W= w_1\otimes E_{i,i} + w_2^*\otimes E_{j,j}+
\sum_{k\notin \{i,j\}} E_{k,k} \in
 \mathcal{U}(\mathbb{M}_n(A))$. Moreover,
 $ WP_{i,j}(v)W^* = P_{i,j}(w_1vw_2)$, therefore by Proposition
 (\ref{pijclass1}) we have $[P_{i,j}(w_1vw_2)]=[P_{i,j}(v)]=[1]$.
\end{proof}

\begin{Proposition}\label{pijclass3} Let $A$ be a unital
$C^*$-algebra. If $u$ and $v$ are self-adjoint unitaries in $A$,
then $[P_{i,j}(uv)]=[1]$ in $K_0(A)$.
\end{Proposition}
\begin{proof} For
$i\neq j$, let $$W= \frac{1}{\sqrt{2}}(uv\otimes E_{i,i} + uv\otimes
E_{i,j}+E_{j,i} - E_{j,j}+\sum_{k\notin \{i,j\}} \sqrt{2}\otimes
E_{k,k})\ ,$$ then $W \in \mathcal{U}(\mathbb{M}_{n}(A))$. Moreover,

\begin{eqnarray*}
W^*P_{i,j}(uv) W &=& \frac{1}{4}(2uv\otimes E_{i,i} +2\otimes
E_{i,j})(\sqrt{2}W)\\
&=& E_{i,i},
\end{eqnarray*}
 and this implies that the projection $P_{i,j}(uv)$ is unitarily
 equivalent to $E_{i,i}$ in $\mathbb{M}_n(A)$, therefore we have that $[P_{i,j}(uv)]=[1]$ in
 $K_0(A)$, hence the proposition has been checked.
\end{proof}

Combining the previous results, we have the following theorem
concerning the $K_0$-class of those projections $P_{i,j}(u)$ in
$\mathcal{P}(\mathbb{M}_n(A))$, evaluated at any unitary $u$ of $A$.

\begin{Theorem}\label{inf} Let $A$ be a simple, unital purely infinite
$C^*$-algebra, such that $K_1(A)$ is the trivial group. If $u\in
\mathcal{U}(A)$, then $[P_{i,j}(u)]=[1]$ in $K_0(A)$.
\end{Theorem}
\begin{proof}
Consider a unitary $u$ of $A$. As $K_1(A)=0$, and we know by
[\cite{Cu1}, p.188] that $K_1(A)\simeq \mathcal{U}(A)/
\mathcal{U}_0(A)$ then using M. Leen's result (Theorem \ref{Leen}),
we have that $u=\prod_{k=1}^n v_k$, where $v_k$ is a self-adjoint
unitary ($*$-symmetry) of $A$. If $n=1$, then the result holds by
using Proposition (\ref{pijclass1}). Proposition (\ref{pijclass3})
proves the case $n=2$. If $n\geq 3$, then the result is done by
Proposition (\ref{pijclass2}), hence the proof is completed.
\end{proof}

Moreover, as M. Broise in [\cite{Br}, Theorem 1] proved that in the
case of von Neumann factors of either type $II_1$ or $III$, the
unitaries are generated by the self-adjoint unitaries, then a
similar result in the case of von Neumann factors can be deduced as
follows:

\begin{Theorem}\label{vn} Let $A$ be a von Neumann factor of type $II_1$ or $III$. If $u\in
\mathcal{U}(A)$, then $[P_{i,j}(u)]=[1]$ in $K_0(A)$.
\end{Theorem}
\begin{proof}
Let $u$ be a unitary of $A$. By [\cite{Br}, Theorem 1], $u$ can be
written as a finite product of self-adjoint unitaries of $A$, then
mimic the proof of Theorem \ref{inf}.
\end{proof}
\noindent Consequently, we have the following results concerning the
$K_0$-class of some certain projections.

\begin{Corollary}
Let $A$ be a unital $C^*$-algebra which is either:\\
(1) Simple, purely infinite, with $K_1(A)=0$, or\\
(2) von Neumann factor of type $II_1$, or $III$.\\
If $u$ be a unitary of $A$, and $p$ is the projection of
$\mathbb{M}_n(A)$ defined by
$$ p= \frac{1}{2}\otimes E_{1,1} + \frac{v}{2}\otimes E_{1,2} + \frac{v^*}{2}\otimes E_{2,1}
+\frac{1}{2}\otimes E_{2,2} + E_{3,3} + E_{4,4} \cdots + E_{m,m} $$
for some positive integer $m\leq n-2$, then $[p]=(m+1)[1]$, in
$K_0(A)$.
\end{Corollary}
\begin{proof}
As the projection $p$ is the orthogonal sums of $P_{1,2}(v) +
E_{3,3} + E_{4,4} \cdots + E_{m,m}$, then by either Theorem
\ref{inf} or \ref{vn},\\
 $$[p]=[1] +([1] + \cdots + [1])= (m+1)[1].$$
\end{proof}

\begin{Corollary}
Let $A$ be a unital $C^*$-algebra which is either:\\
(1) Simple, purely infinite, with $K_1(A)=0$, or\\
(2) von Neumann factor of type $II_1$, or $III$.\\
If $v_1, v_2 \cdots v_n$ are unitaries of $A$, and $p$ is the
projection of $\mathbb{M}_{2n}(A)$ defined by
\begin{eqnarray*}
p &=&\frac{1}{2}\otimes E_{1,1} + \frac{v_1}{2}\otimes E_{1,2} +
\frac{v_1^*}{2}\otimes E_{2,1} + \frac{1}{2}\otimes E_{2,2}\\
 &+& \frac{1}{2}\otimes E_{3,3} + \frac{v_2}{2}\otimes E_{3,4} +
\frac{v_2^*}{2}\otimes E_{4,3} + \frac{1}{2}\otimes E_{4,4} + \cdots
\\ &+& \frac{1}{2}\otimes E_{2n-1,2n-1} + \frac{v_n}{2}\otimes
E_{2n-1,2n} + \frac{v_n^*}{2}\otimes E_{2n,2n-1} +
\frac{1}{2}\otimes E_{2n,2n},
\end{eqnarray*} then $[p]=n[1]$, in $K_0(A)$.
\end{Corollary}
\begin{proof} Using Theorem
\ref{inf} (or Theorem \ref{vn}), we have
$$[p]= [P_{1,2}(v_1)]+ [P_{3,4}(v_2) + \cdots + [P_{2n-1,2n}(v_n)]=n[1].$$
\end{proof}

\noindent Now let us prove the following lemma, which will be used
in order to prove our main result in this section (Theorem
\ref{mainresultforCuntz}), which is in fact a consequence
application of Theorem \ref{unitaryfactor}, to the case of Cuntz
algebras $\mathcal{O}_n$.

\begin{Lemma}\label{class-eta} Let $A$ be a unital, simple purely infinite $C^*$-algebra, with $K_1(A)=0$,
 and let $\{e_{i,j}\}^n$, with $e_{1,1}\sim 1$ be a
system of matrix units of $A$ . Then for any unitary $u\in
\mathcal{U}(A)$ we have $[\eta(P_{i,j}(u))]=[1]$ in $K_0(A)$.
\end{Lemma}
\begin{proof} As we have seen in the proof of Propositions \ref{pijclass1}, \ref{pijclass2}, \ref{pijclass3} and Theorem \ref{inf}, there exists
 a unitary $W\in \mathcal{U}(\mathbb{M}_n(A))$, such that
$W^*P_{i,j}(u)W=E_{i,i}$. Therefore,
\[ \eta(W)^*\eta(P_{i,j}(u))\eta(W)=\eta(E_{i,i})= \eta_1\hat{\Delta}_v (E_{i,i})=\eta_1(e_{1,1}\otimes E_{i,i})=e_{i,i}.\]
Then
\[\eta(P_{i,j}(u))\sim_u e_{i,i}\sim e_{1,1}\sim 1,\]
hence $\eta(P_{i,j}(u))$ and $1$ have the same class in $K_0(A)$.
\end{proof}

Finally, let us consider the case of the Cuntz algebra
$\mathcal{O}_n$. Let $u$ be a self-adjoint unitary (involution), so
$u=1-2p$, for some $p\in \mathcal{P}(\mathcal{O}_n)$. We recall the
concept \textit{type of involution} which is introduced by the
author in \cite{Ahmed2}, as follows: Since $K_0(\mathcal{O}_n)\simeq
\mathbb{Z}_n$ (see \cite{Cu1}), then the type of $u$ is defined to
be the element $[p]$ in $K_0(\mathcal{O}_n)$. By (\cite{Ahmed2},
Lemma 2.1), two involutions are conjugate as group elements in
$\mathcal{U}(\mathcal{O}_n)$ iff they have the same type.

 As a consequence of
Theorem \ref{unitaryfactor}, and the results concerning the
$K_0$-group of the projections $P_{i,j}(u)$, which are deduced in
this section, we have the following result.

\begin{Theorem}\label{mainresultforCuntz}
If $u$ is a unitary of $\mathcal{O}_n$, then there exist
self-adjoint unitaries $z_1,z_2,z_3$ and $v_k$, for $1\leq k\leq 8$
such that
\begin{equation}\label{Unitary-decom-Cuntz}
u=z_1(\prod_{k=1}^8v_k)z_2z_3,
\end{equation}
  $v_k\in
\{1-2\eta P_{i,j}(\omega)\},\ \omega \in \mathcal{U}(\mathcal{O}_n)$
consequently, all the $v_k$ factors are conjugate involutions.

\end{Theorem}
\begin{proof}
Using \cite{Cu1} and \cite{Cu2}, the Cuntz algebra $\mathcal{O}_n$
is simple, unital purely infinite $C^*$-algebra with trivial
$K_1$-group. Then the decomposition of $u$ as in Equation
\ref{Unitary-decom-Cuntz} holds by Theorem \ref{unitaryfactor}, so
the type of each involution $v_k$ is $[\eta(P_{i,j}(w))]$, for some
$1\leq i\neq j \leq n$ and a unitary $w$, hence by Lemma
\ref{class-eta}, the type of $v_k$ is $1$. Then by [\cite{Ahmed2},
Lemma 2.1], all these involutions are conjugate indeed, to the
trivial involution $-1$.
\end{proof}

\end{document}